\documentclass[review]{elsarticle}
\usepackage{lineno,hyperref}
\usepackage{amssymb}
\usepackage{amsmath}
\usepackage{amsthm}
\usepackage{dcolumn}
\usepackage{endnotes}
\usepackage{enumerate}
\usepackage{soul, color, xcolor}
\usepackage{tabularx}
\usepackage[matrix,arrow]{xy}

\journal{Topology and its Applications}

\newtheorem{theorem}{Theorem}[section]
\newtheorem{proposition}[theorem]{Proposition}
\newtheorem{lemma}[theorem]{Lemma}
\newtheorem{corollary}[theorem]{Corollary}

\theoremstyle{definition}
\newtheorem{definition}[theorem]{Definition}
\newtheorem{example}[theorem]{Example}

\newcommand{\ua}{\mathord{ \uparrow}}
\newcommand{\da}{\mathord{\downarrow}}
\newcommand\twoheaduparrow{\mathord{\rotatebox[origin=c]{90}{$\twoheadrightarrow$}}}
\newcommand\twoheaddownarrow{\mathord{\rotatebox[origin=c]{90}{$\twoheadleftarrow$}}}

\newcommand{\dda}{\twoheaddownarrow}
\newcommand{\dua}{\twoheaduparrow}

\newcommand{\mjx}{\mathcal J(x)}
\newcommand{\mj}{\mathcal J}
\newcommand{\mi}{\mathcal I}
\newcommand{\mn}{\mathbb N}
\newcommand{\Max}{\text{Max}}

\begin{document}

\begin{frontmatter}

\title{The answers to two problems on maximal point spaces of domains}
\tnotetext[mytitlenote]{This work was supported by the National Natural Science Foundation of China (No.12101313, 62476030).} 

%
%

\author[1]{Xiaoyong Xi}
\author[2,22]{Chong Shen\corref{cor2}} 
\ead{shenchong0520@163.com}
\cortext[cor2]{Corresponding author.}
\author[3]{Dongsheng Zhao}

\address[1]{School of Mathematics and Statistics,	Yancheng Teachers University, Yancheng 224002, Jiangsu,  China}
\address[2]{School of Science,	Beijing University of Posts and Telecommunications,  Beijing 100876,	China}
\address[22] {Key Laboratory of Mathematics and Information Networks (Beijing University of Posts and Telecommunications), Ministry of Education, China.}
\address[3]{Mathematics and Mathematics Education, National Institute of Education,
	Nanyang Technological University,  1 Nanyang Walk, 637616, Singapore}

\begin{abstract}
	A topological space  is domain-representable (or, has a domain model)  if it is homeomorphic to the maximal point space $\mbox{Max}(P)$ of a domain $P$ (with the relative Scott topology). 
We first construct an example to show that the set of  maximal points of an ideal domain $P$ need not be a $G_{\delta}$-set in the Scott space $\Sigma P$, thereby answering  an open problem from Martin (2003). In addition, 
 Bennett and Lutzer (2009) asked whether $X$ and $Y$ are domain-representable if their product space $X \times Y$ is domain-representable. This problem was first solved by \"{O}nal and Vural (2015).
In this paper, we  provide a new approach to Bennett and Lutzer's problem.
 \end{abstract}


\end{frontmatter}


\section{Introduction}
The Scott topology of posets, introduced by Dana Scott \cite{Scott1,Scott2}, is one of the core structures in domain theory.
Spaces that are homeomorphic to the space of maximal elements of a domain (i.e., a continuous dcpo) with the relative Scott topology are referred to as domain-representable by Bennett and Lutzer
in \cite{1}, whereas Martin, in several of his papers (e.g., \cite{4}), 
refers to such spaces as having a domain model.
The concept of representing a topological space as the set of maximal elements of a specific type of poset originates from theoretical computer science \cite{12}, which also proves to be highly valuable in the study of completeness properties of Baire spaces. 
Identifying domain-representable spaces has become one of the active areas in domain theory. Initially, these spaces were named by various terms such as ``$P$-complete" or ``FY-complete" by different authors, but the term "directed complete" was eventually introduced in \cite{2}.  Fleissner and Yengulalp provided a simplified definition of domain-representability in \cite{3}, proving its equivalence to the traditional definition. 

One  classical  example is the poset $\mathrm{IR}(\mathbb{R})$ of all closed intervals of reals, with the inverse set inclusion order, is a domain model of the Euclidean space $\mathbb{R}$.  
 Lawson confirmed that a space  is Polish if and only if it has an $\omega$-domain model satisfying the Lawson condition \cite{ref1}. 
 K. Martin \cite{ref3} proved  that the space $\mbox{Max}(P)$ of an $\omega$-continuous dcpo $P$ is regular if and only if it is Polish.
 The formal balls of a metric space, initially introduced by Weihrauch and Schreiber as a domain-theoretic representation of metric spaces \cite{Weihrauch}, were further employed  by Edalat and Heckmann to prove that every complete metric space is domain-representable \cite{edalat-1998}. The Sorgenfrey line, which is not metrizable, is also domain-representable \cite{Bennett-Lutzer-2009}. 
 A more general result is that every  $T_1$ space is homeomphic to the maximal points of some dcpo  with the relative Scott topology \cite{zhao-xi-2018}.

 As pointed out by Martin \cite{ref2}, knowing that the set of maximal points  is a $G_{\delta}$-subset is often useful in proofs. For instance, Edalat studied the connection between measure theory and the probabilistic powerdomain \cite{edalat-1998} assuming a separable metric space that embedded as a $G_{\delta}$-subset of a countably based domain.
  
The following are some results on $G_{\delta}$-subsets $E \subseteq \mbox{Max}(P)$ obtained by  Martin:
 \begin{itemize}
 	\item [(i)] If $P$ is an $\omega$-domain and ${{\rm Max}(P)}$ is regular, then ${{\rm Max}(P)}$ is a $G_{\delta}$-set \cite{ref3}.
 	\item [(ii)] For any $\omega$-ideal domain $P$ (see Section 1 for the definition of $\omega$-ideal domain), ${{\rm Max}(P)}$ is a $G_{\delta}$-set \cite{ref2}.
 \end{itemize}
Furthermore, Brecht, Goubault-Larrecq, Jia, and Lyu  \cite{Brecht} showed that continuous valuations on $G_\delta$-subsets of locally compact sober spaces can be extended to measures. They also demonstrated that each $G_\delta$-subset of the maximal point space of an $\omega$-ideal domain is a quasi-Polish space. 

This paper focuses on two problems related to maximal point spaces of domains. The first problem was posed  by Martin (\cite[Section 8, Ideas (iv)]{ref2}):
\begin{itemize}
	\item Is there an ideal domain whose maximal points do not form a $G_\delta$-set? 
\end{itemize}
In Section 3, we construct an ideal domain $P$ such that $\Max (P)$ is not  $G_{\delta}$-set, thereby providing a negatve answer.

The second  problem was posed by Bennett and Lutzer in \cite{Bennett-Lutzer-2009}:
\begin{itemize}
	\item Suppose the product space $X \times Y$ is domain-representable, where $Y \neq \emptyset$. Must  $X$ be  domain-representable?
\end{itemize}
In \cite{5}, \"{O}nal and Vural proved that domain-representability is hereditary with respect to retracts, meaning that retracts of a domain-representable space are also domain-representable. Consequently, if the product of two spaces is domain-representable, the factor spaces must also be domain-representable; thus gives a positive answer to Bennett and Lutzer's problem.
In Section 4, we provide a new proof for the above problem, which is distinct from that of \cite{5} in that it relies directly on the original definition of domain-representable spaces proposed by Martin in \cite{4}.
Our method may help one solve  the corresponding product problem of Scott-domain-representable spaces.

 \section{Preliminaries}

 We first review  some basic concepts and notations that will be used later. For more details, we refer the reader to \cite{redbook,goubault}.

 Let $P$ be a poset.
 For a subset $A$ of  $P$, we shall adopt the following standard notations:
 $$\ua A=\{y\in P: \exists x\in A, x\leq y\},\ \da A=\{y\in P:\exists x\in A, y\leq x\}.$$
 For each $x\in X$, we simply write $\ua x$ and $\da x$ for $\ua\{x\}$ and  $\da \{x\}$, respectively. A subset $A$ of $P$ is called a \emph{lower} (\emph{upper}, resp.) \emph{set} if $A=\da A$ ($A=\ua A$, resp.).
 An element $x$ is \emph{maximal}  in $P$, if  $P\cap \ua x=\{x\}$. The set of all maximal elements of $P$ is denoted by $\Max (P)$.

 A nonempty subset $D$ of $P$ is \emph{directed}  if every two
 elements in $D$ have an upper  bound in $D$.  For  $x,y\in P$,  $x$ is \emph{way-below} $y$, denoted by $x\ll  y$, if for each directed subset
 $D$ of $P$ with  $\bigvee D$ existing,
 $y\leq\bigvee D$
 implies $x\leq d$ for some $d\in D$.  Denote $\dua x=\{y\in P:x\ll y\}$ and $\dda x=\{y\in P:y\ll x\}$. A poset $P$ is
 \emph{continuous}, if for each $x\in P$, the set $\dda x$ is directed and
 $x= \bigvee\dda x$.

 An element $a\in P$ is called \emph{compact}, if $a\ll a$. The set of all compact elements of $P$ will be denoted by   $K(P)$.
 Then, $P$ is called \emph{algebraic} if for each $x\in P$, the set $\{a\in K(P):a\leq x\}$ is directed and
 $x=\bigvee\{a\in K(P):a\leq x\}$.
 A continuous (algebraic, resp.) dcpo  is also called a \emph{domain} (an \emph{algebraic domain}, resp.).
 A subset $B\subseteq P$ is a {\em base} of  $P$ if  for each $x\in P$, $B\cap \dda x$ is directed and $\bigvee(B\cap \dda x)$=$x$.
 If $P$ has a countable base, then $P$ is called an {\em $\omega$-continuous dcpo} or \emph{$\omega$-domain}.

 A subset $U$ of  $P$ is \emph{Scott open} if
 (i) $U=\mathord{\uparrow}U$ and (ii) for each directed subset $D$ of $P$ for
 which $\bigvee D$ exists, $\bigvee D\in U$ implies $D\cap
 U\neq\emptyset$. All Scott open subsets of $P$ form a topology on $P$,
 called the \emph{Scott topology} and denoted by $\sigma(P)$. The space $\Sigma P=(P,\sigma(P))$ is called the
 \emph{Scott space} of $P$.

 Let $X$ be a topological space. A  subset $G \subseteq X$ is called  a \emph{$G_{\delta}$-set}, if there exists a countable family $\{U_n:n\in \mathbb N\}$ of open sets such that  $G=\bigcap_{n\in\mathbb{N}} U_{n}$.

 \begin{definition}[\cite{ref2}]
 	A domain is called \emph{ideal} if every element is either compact or maximal.
 \end{definition}

A poset model of a space $X$ is a poset $P$ such that $X$ is homeomorphic to $\Max (P)$ with the relative Scott topology on $P$.
 If $P$ is a poset model of $X$ such that $P$ is an (ideal) domain, then we say that $X$ has an (ideal) domain model, or equivalently, $X$ is (ideal)
 domain-representable.

\section{The set of maximal points  of	an ideal domain need not be a $G_{\delta}$-set}

In \cite[Section 8, Ideas (iv)]{ref2}, Martin asks: \begin{itemize}
	\item Is there an ideal domain whose  maximal elements do not form a  $G_\delta$-set?
\end{itemize}
The following example shows that the maximal point set of an ideal domain need not be a $G_{\delta}$-set.

\begin{example}\label{exm2}
	Let $\mn$ be the set of all natural numbers equipped with the usual order.
Let $E=\mn\times(\mn\cup \{\infty\})$, and $F=\prod_{i\in\mn}C_i$ be the Cartesian product of all $C_j$, where
	$C_i= \{i\}\times\mn$ for each $i\in\mn$.
	For each $\varphi\in F$, $\varphi(i)\in C_i$ denotes the $i$-th coordinate of $\varphi$. It is clear that $F$ is uncountable and $\bigcup_{i\in\mn}C_i\subseteq E$.
	
	Let $L=E\cup (F\times \{0,1\})$ and define a binary relation $\leq$ on $L$ as follows:
	\begin{itemize}
		\item [(i)] $\forall i\in\mn$, $(i,m)\leq (i,n)\leq (i,\infty)$ whenever $m\leq n$ in $\mn$;
		\item [(ii)] $\forall \varphi\in F$, $\forall i\in \mn$, $\varphi(i)\leq (\varphi,0)\leq(\varphi,1)$.
	\end{itemize}
	Then, one can easily verify that $(L,\leq)$ is an ideal domain, as shown in Figure \ref{fig3}.
	
	\begin{figure}[ht] 
		\centering 
		\includegraphics[width=0.7\textwidth]{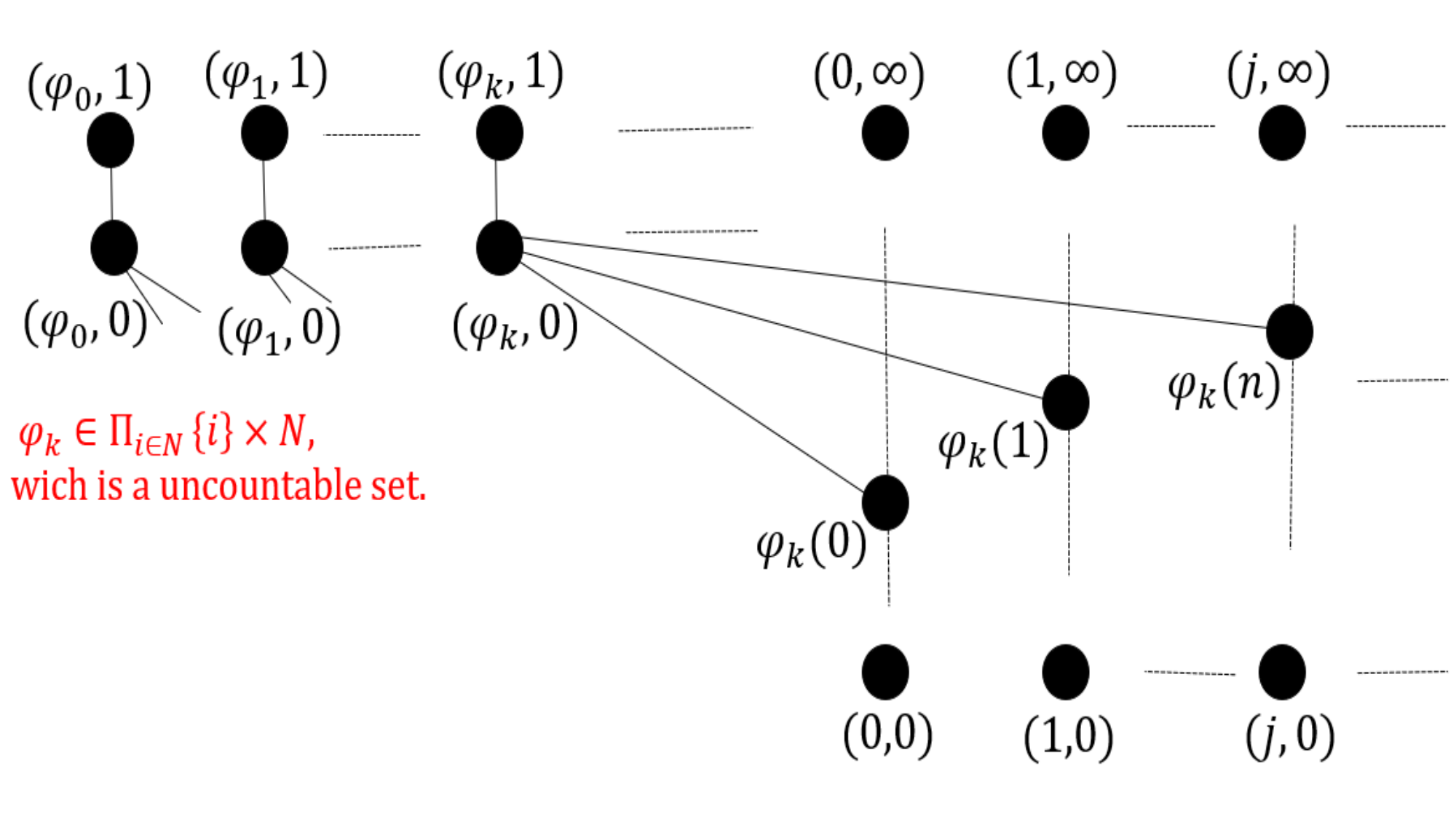} 
		\caption{The ideal domain $L$} 
		\label{fig3} 
	\end{figure}

	{\bf Claim: } $\Max(L)$ is not a $G_\delta$-set.
	
	Suppose that $\Max(L)$ is a $G_\delta$-set. Then there is a countable family  $\{U_i: i\in\mathbb{N}\}$ of Scott open subsets of $L$ such that  $\Max(L)=\bigcap_{i\in\mn} U_i$.
	For each $j\in\mn$, since
	$$\bigvee\{(j,n):n\in\mn\}=(j,\infty)\in\bigcap_{i=0}^{\infty} U_i\subseteq U_j,$$
	there exists $n_j\in\mathbb N$ such that  $(j,n_j)\in U_j$ (clearly $(j,n_j)\in C_j$).
	Define $\varphi\in F$ by $\varphi(j)=(j,n_j)\in  U_j$ for each $j\in\mathbb{N}$. Then, since $(\varphi,0)\geq\varphi(j)\in U_j$ for any $j\in \mn$, we have that
	$(\varphi,0)\in\bigcap_{j\in\mn}U_j=\Max(L)$, but $(\varphi,0)\notin \Max(L)$, which is a contradiction.  Therefore, 	$\Max(L)$ is not a $G_\delta$-set.
\end{example}

It is worth noting that though the above maximal point space $\Max(L)$ is not a $G_{\delta}$-set in $\Sigma L$, but it  is  homeomorphic to a $G_{\delta}$-subspace of another ideal domain $\widehat{L}$, where $\widehat{L}$ is a slight modification of $L$. This is demonstrated in the following example.

\begin{example}\label{exam3}
	Let \(\widehat{L} = L \setminus \{(\varphi, 1): \varphi \in F\}\) with the inherited order from \(L\), where \(F\) and \(L\) are sets defined in Example \ref{exm2}.
	Then, $\widehat{L}$ is an ideal domain, as can be seen in Figure \ref{fig4}. 
	\begin{figure}[ht] 
		\centering 
		\includegraphics[width=0.7\textwidth]{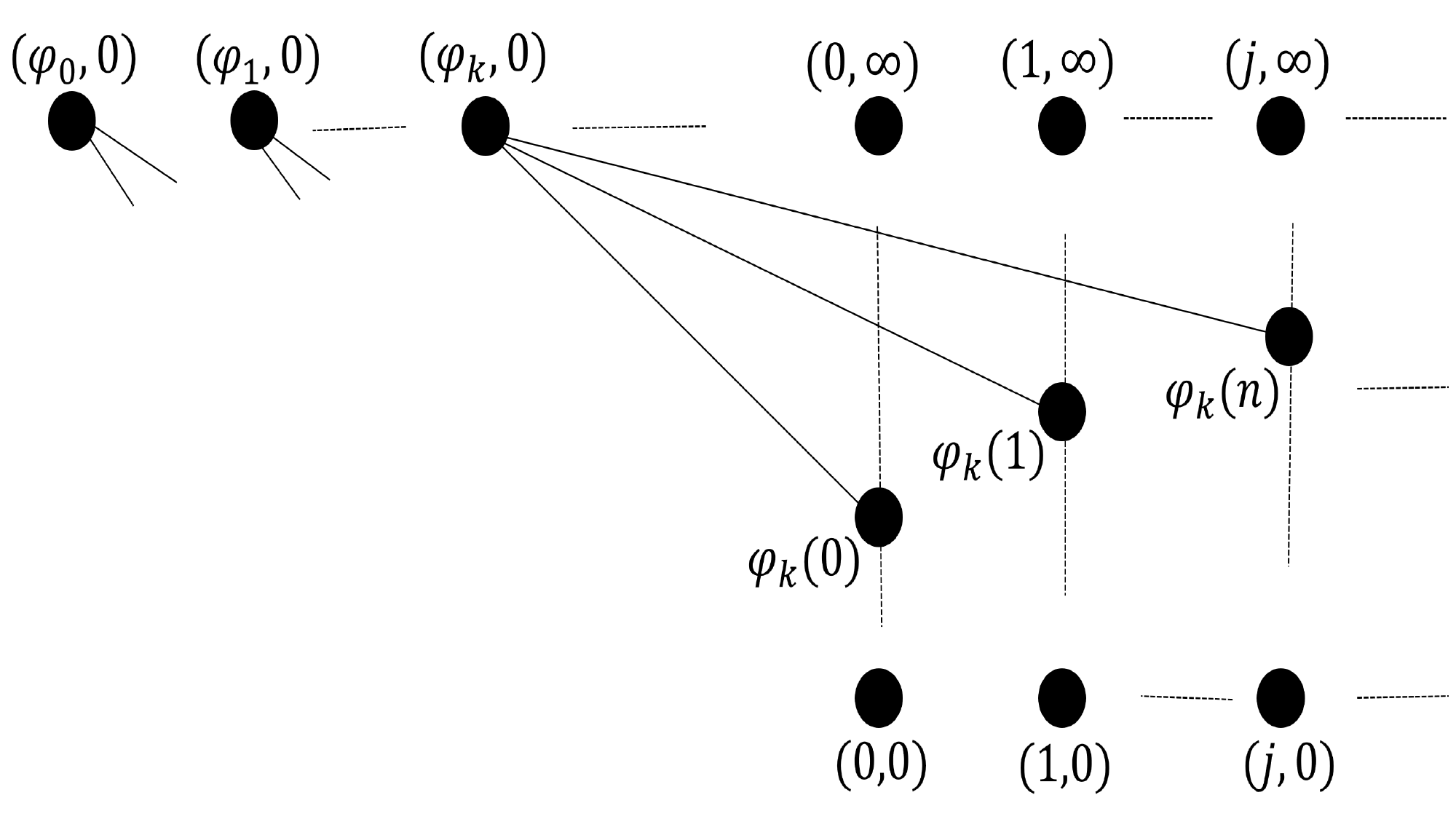}
		\caption{The ideal domain $\widehat{L}$}
		\label{fig4}
	\end{figure}

	{\bf Claim:} $\Max(\widehat{L})$ is a $G_{\delta}$-set.
	
	For each $k\in\mn$, define
	$$U_k:=\widehat{L}\setminus\bigcup\{\da(i,k):i\leq k\}.$$
	Then, it is clear that $\{U_k:k\in\mn\}$ is a $G_{\delta}$-set and such that 
	$\bigcap_{k\in\mn}U_k=\Max(\widehat{L})$. Therefore, $\Max(\widehat{L})$ is a  $G_{\delta}$-set.
\end{example}

\section{The productivity of domain-representable spaces}

In this section, we provide a direct construction of domain models for the factor spaces $X$ and $Y$, assuming that the product space $X \times Y$ is domain-representable. This offers another proof to the problem posed by Bennett and Lutzer in \cite{Bennett-Lutzer-2009}, as mentioned in the Introduction.  
It should be noted that this construction is not straightforward and cannot be derived directly, as the following example demonstrates:
\begin{example}\label{exam1}
	Let $P=(\mathbb N\times \{0,1\})\cup \mathbb N\cup\{\infty\}$ equpped with the partial order given by
	\begin{itemize}
		\item [(i)]  $\infty\leq (0,1)$;
		\item [(ii)] $\forall n\in\mn$, $n\leq n+1\leq \infty$.
	\end{itemize}
	It is clear that $P$ is a domain model of $\mathbb N\times \{0,1\}$ with the discrete topology. 
Now, if we want to construct a domain model for $\mathbb{N} \times \{0\}$, which is homeomorphic to the factor space $\mathbb{N}$ (which is also the discrete space), a natural candidate would be  $Q = \da (\mathbb{N} \times \{0\})$ with the inherited order, which indeed is its poset model in this example. However, this is not  a dcpo. Therefore, $Q$ is not a domain model of $\mathbb N$ with the discrete order.
	\begin{figure} 
	\centering 
	\includegraphics[width=0.81\textwidth]{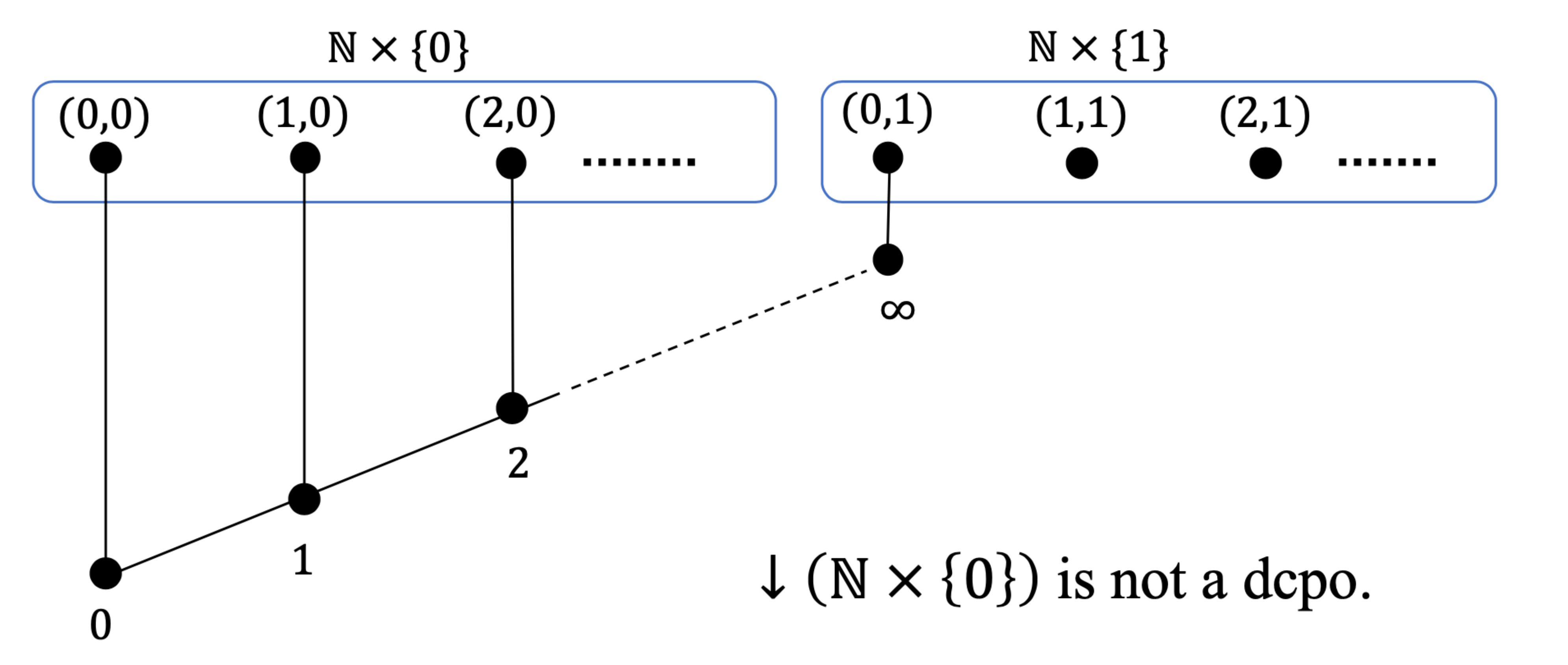} 
	\caption{The model of $\mathbb N\times \{0,1\}$ with the discrete topology} 
	\label{fig5} 
\end{figure}	
\end{example}

In Example \ref{exam1}, one can easily observe that $\da (\mathbb N\times \{1\})$ is a domain model of $\mathbb N$ with the discrete topology.
As a supplementary conclusion to the example, we add the following propositions.
\begin{proposition}
Let $P$ be a (Scott-) domain model of the product $X\times Y$ of two spaces $X$ and $Y$. If there exists $y\in Y$ such that 
$\da (X\times \{y\})$ is a Scott closed subset (or equivalently, a subdcpo) of $P$, then, with the inherited order, it is a (Scott-) domain model of $X$. 
\end{proposition}
\begin{proof}
The proof is clear by using the following facts:
\begin{itemize}
	\item [(i)] The Scott topology on $\da (X\times \{y\})$ agrees with the relative Scott topology from $P$ (see Exercise II-1.26 on page 151 in \cite{redbook}. \cite{redbook}).
	\item [(ii)] Every Scott subset of a domain is a closure system, hence is a domain (see Corollary I-2.5 in \cite{redbook}).
	\item [(iii)] It is easy to check that every lower subset of a bounded complete dcpo is also bounded complete.
\end{itemize}
\end{proof}

\begin{lemma}\label{lem}
	Let $P$ be an ideal domain. If  $X\subseteq \Max(P)$ is closed  in the maximal point space, then $\da X$  is a Scott closed subset of $P$, and hence is  an ideal domain.
\end{lemma}
\begin{proof}
	We first show that $\da X$ is a Scott closed subset of $P$.
	Suppose $D$ is a directed subset of $\da X$. Then, since $P$ is an ideal domain, there are two cases:
	\begin{itemize}
		\item [(c1)] $\bigvee D$ is compact. This means $\bigvee D\in D\subseteq \da X$.
		\item [(c2)] $\bigvee D$ is a maximal point of $P$. Since $X$ is closed in $\Max (P)$, there exists a Scott closed subset $C$ of $P$ such that $X=C\cap \Max (P)$, which means $\da X = \da (C\cap \Max(P))$. Then,  $D\subseteq \da X = \da (C\cap \Max(P))\subseteq \da C=C$, which follows that $\bigvee D\in  C\cap \Max(P)= X$.
	\end{itemize}
Thus, $\da X$ is Scott closed. It is trivial to check that every Scott closed subset of an ideal domain is an ideal domain.
\end{proof}

From Lemma \ref{lem}, one can deduce the following corollary.
\begin{corollary}
Let $P$ be an ideal domain model of the product $X\times Y$ of two spaces $X$ and $Y$. Then, for each $y\in Y$,  $\da (X\times\{y\})$ is a domain model of $X$. 
\end{corollary}

Note that the above corollary does not hold for Scott domains.
For example, consider the $P$ in  Example \ref{exam1}, which serves as a Scott domain model for $\mn\times \{0,1\}$ with the discrete topology,  but $\da(\mathbb N\times \{0\})$ is not a Scott domain model of $\mn$ with the discrete topology.

Next, we provide a new construction method for the domain model of factor spaces. Before doing so, we state a lemma, which is well-known, will be used in our proof.

\begin{lemma}[\cite{Bennett-Lutzer-2009}]\label{lem1}
If  $X$ has a  domain model, then it has an algebraic domain model.
\end{lemma}

\begin{theorem}
If the product $X\times Y$ of two spaces $X$ and $Y$ is domain-representable, then $X$ and ${Y}$ are domain-representable. 
\end{theorem}		
\begin{proof}
According to Lemma \ref{lem1}, assume that $(P, \leq)$ is an algebraic domain model of $X\times Y$.
Without loss of generality, we may assume that the  maximal point space $\Max(P)$ equals $X\times Y$, that is,
\begin{center}
$\Max(P)=X\times Y$ and
$\mathcal O(X\times Y)=\{U\cap \Max(P):U\in\sigma(P)\}$.	
\end{center}
Let $y_0$ be a fixed element of $Y$.
To construct a domain model of $X$, it suffices to construct a domain model of the subspace $X\times\{y_0\}$ of $X \times Y$. We do this in three steps.

\medskip

	{\bf Step 1.} Construct a poset $Q$.
	
	Let
	$$ Q = \{(U,V,k)\in\mathcal O^*(X)\times\mathcal{U}(y_0)\times K(P): U\times V \subseteq \ua k \cap \Max(P)\},$$
	where $\mathcal U(y_0)$ is the set of all open neighborhoods of $y_0$.
	The elements of $Q$ are illustrated by Figure \ref{fig1}.
	\begin{figure}[ht] 
		\centering 
		\includegraphics[width=1\textwidth]{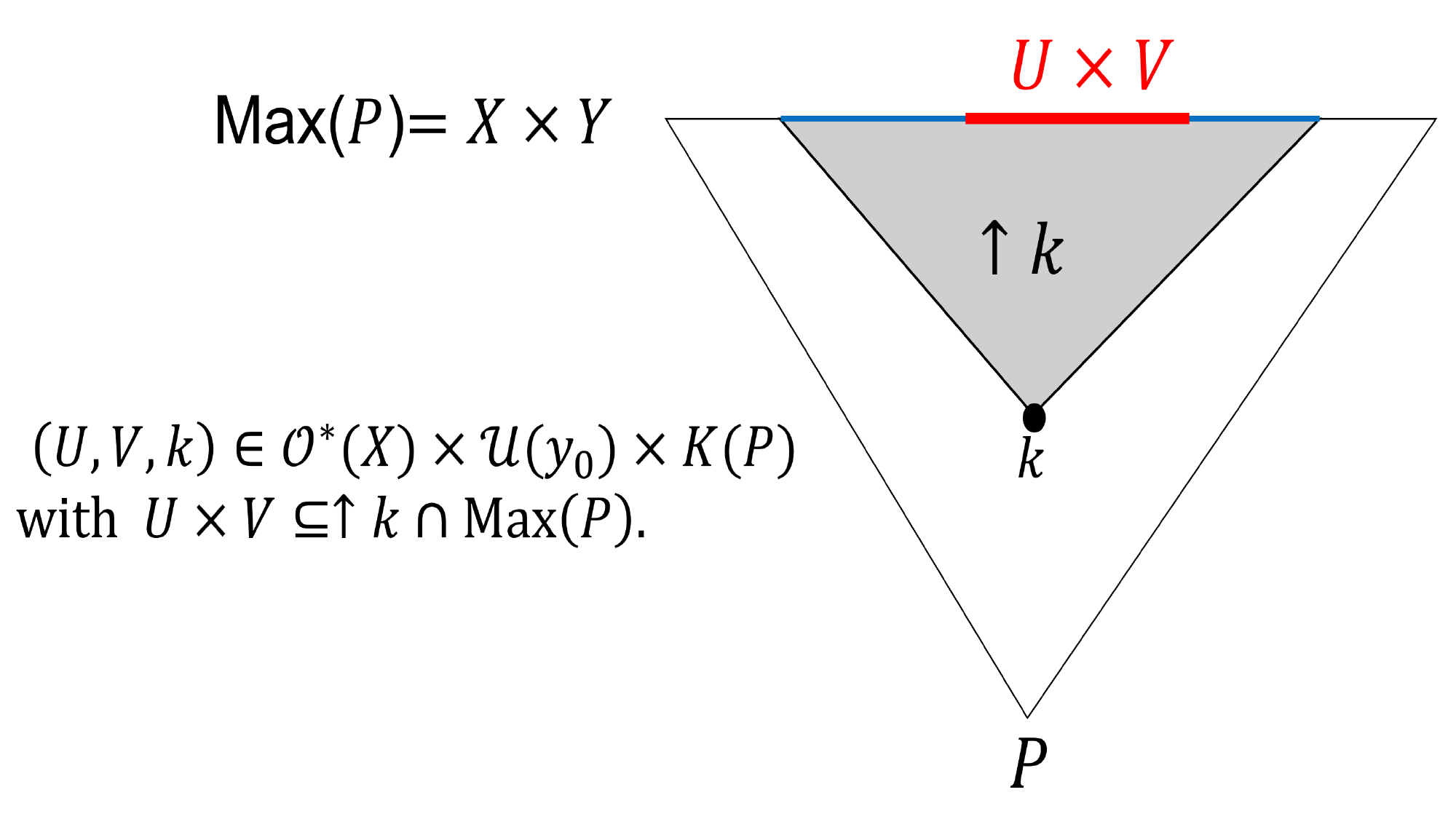} 
		\caption{The elements in $Q$} 
		\label{fig1} 
	\end{figure}
Define a binary relation $\sqsubseteq$ on $ Q $ by $(U_1,V_1,k_1)\sqsubseteq(U_2,V_2,k_2)$ iff it satisfies the following two conditions:
\begin{itemize}
	\item [(i)] $k_1 \leq k_2$;
	\item [(ii)] $\ua k_2 \cap \Max(P)\subseteq U_1\times V_1$,
\end{itemize}
as illustrated in  Figure \ref{fig2}.

		{\bf Claim 1.} $\sqsubseteq$ is a partial order on $Q$.
		
		\begin{figure}[ht] 
			\centering 
			\includegraphics[width=1\textwidth]{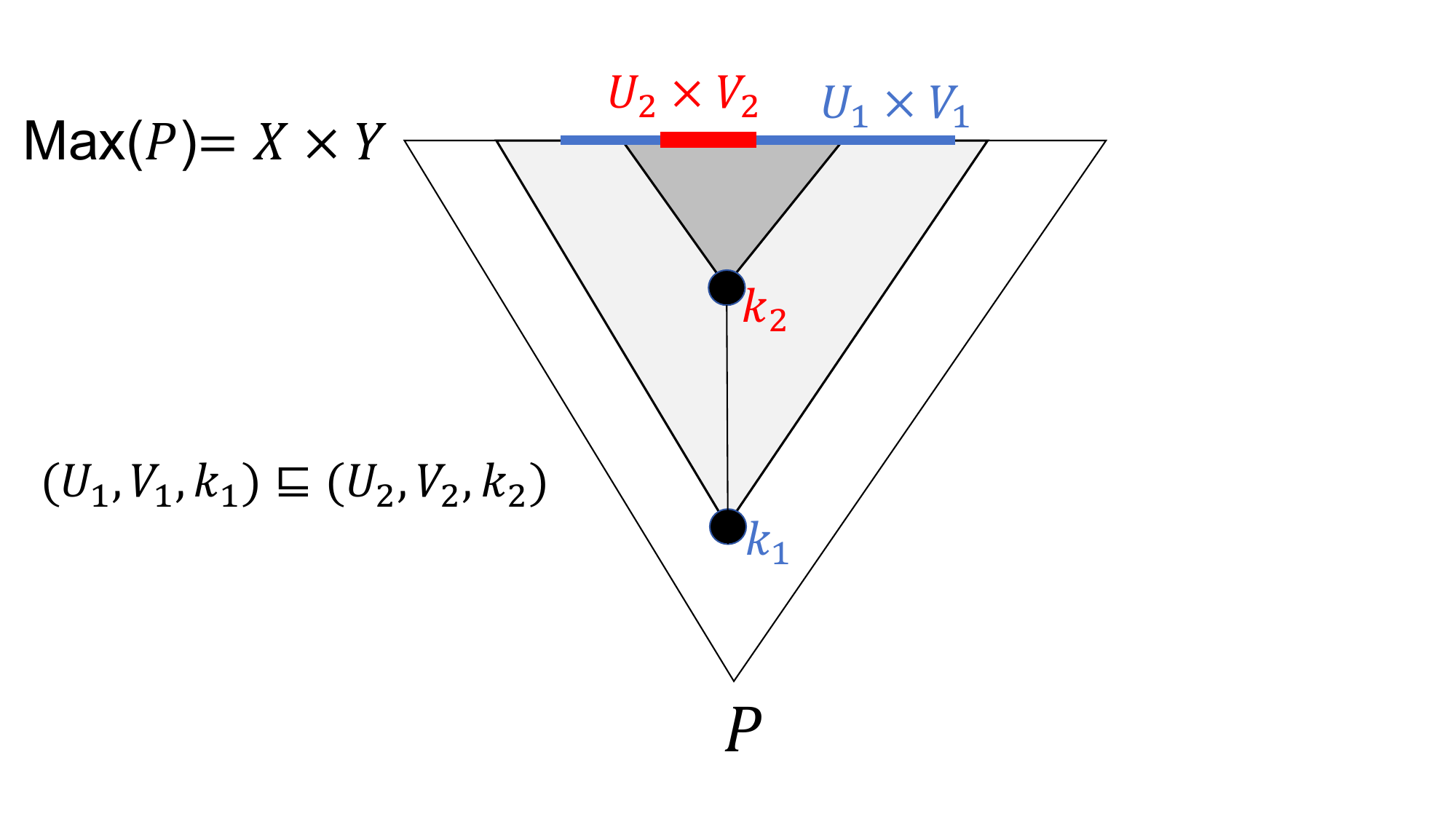} 
			\caption{The  poset $Q$} 
			\label{fig2} 
		\end{figure}
		
		Reflexivity is clear. 
		To show the antisymmetry, suppose  $(U_1,V_1,k_1)\sqsubseteq (U_2,V_2,k_2)$ and $(U_2,V_2,k_2)\sqsubseteq (U_1,V_1,k_1)$.
		Then, on the one hand, it follows that
		$k_1\leq k_2$ and $k_2\leq k_1$, so $k_1=k_2$. On the other hand,
		by the definition of $Q$ and the relation $\sqsubseteq$,  we have that
		$$\ua k_2\cap \Max(P)\subseteq U_1\times V_1\subseteq  \ua k_1\cap \Max(P)
		\subseteq U_2\times V_2\subseteq 	 \ua k_2 \cap \Max(P),$$
which follows that $U_1\times V_1=U_2\times V_2$, that is, $U_1=U_2$ and $V_1=V_2$. Therefore, 	 $(U_1,V_1,k_1)= (U_2,V_2,k_2)$.

Now we verify the transitivity. Suppose $(U_1,V_1,k_1)\sqsubseteq (U_2,V_2,k_2)\sqsubseteq (U_3,V_3,k_3)$. On the one hand, 
we have that $k_1\leq k_2\leq k_3$, so $k_1\leq k_3$. On the other hand, it follows that
$$\ua k_3 \cap \Max(P)\subseteq U_2\times V_2\subseteq \ua k_2 \cap \Max(P)\subseteq U_1\times V_1,$$ 
and consequently, $(U_1,V_1,k_1)\sqsubseteq (U_3,V_3,k_3)$. 
Therefore, $\sqsubseteq$ is a partial order.

\medskip

			{\bf Step 2.}
			Let $P_X=(Idl(Q),\subseteq)$, where $Idl(Q)$ denotes the set of all ideals of $Q$.
Then, $P_X$ is an algebraic domain whose compact elements are the principal ideals (see \cite[Proposition I-4.10]{redbook}).
			Next, we
			show that $\Max(P_X)=\{\mjx:x\in X\}$,
		 where
		$$\mathcal{J}(x)\triangleq \{(U,V,k)\in Q: x\in U\}.$$
		\medskip
		
{\bf Claim 2.}  For any $x\in X$, $\mj(x)\in P_X$.

We first show that 	 $\mathcal J(x)$ is a lower set.
 Suppose $(U,V,k)\sqsubseteq(U_0,V_0,k_0)\in \mjx$.
Then, we have that
			$$(x,y_0)\in U_0\times V_0\subseteq  \ua k_0 \cap \Max(P) \subseteq U\times V,$$
			which follows that $x\in  U$. This means that  $(U,V,k)\in \mjx$.
			
Now we show that $\mjx$ is directed.  That $\mjx\neq\emptyset$ is clear. 
			Suppose  $(U_1,V_1,k_1),(U_2,V_2,k_2)\in \mjx$.
			By the definition of $\mjx$ and $Q$, we have that
			 $$(x,y_0)\in (U_1\times V_1)\cap (U_2\times V_2)=(U_1\cap U_2)\times (V_1\cap V_2)\in \mathcal O(X\times Y)=\mathcal O(\Max(P)),$$
 so there exists $k_3\in K(P)$ such that
$$(x,y_0)\in \ua k_3\cap \Max(P)\subseteq (U_1\times V_1)\cap (U_2\times V_2)\subseteq\ua k_1\cap \ua k_2\cap \Max(P).$$
It follows that $k_1,k_2,k_3\leq (x,y_0)$, and since
$P$ is algebraic, there exists
$k_0\in K(P)$ such that
$k_1,k_2,k_3\leq k_0\leq (x,y_0)$. We have  that
$$(x,y_0)\in \ua k_0\cap \Max(P)\subseteq (U_1\times V_1)\cap (U_2\times V_2)\subseteq\ua k_1\cap \ua k_2\cap \Max(P).$$
On the other hand,
 since	$(x,y_0)\in \ua k_0\cap \Max(P)\in\mathcal O(X\times Y)$ and $\mathcal O(X)\times \mathcal O(Y)$ forms a basis of $\mathcal O(X\times Y)$,  there exist $U_0\in \mathcal (X)$ and $V_0\in\mathcal O(Y)$ (we may assume $U_0\subseteq U_1\cap U_2$ and $V_0\subseteq V_1\cap V_2$) such that
		$$(x,y_0)\in U_0\times V_0\subseteq \ua k_0\cap \Max(P).$$
It follows that $(U_1,V_1,k_1), (U_2,V_2,k_2)\sqsubseteq(U_0,V_0,k_0)\in\mjx$. Therefore, $\mj(x)$ is an ideal of $Q$, i.e., $\mjx\in P_X$.

\medskip
			
	{\bf Claim 3.} $\Max(P_X)\subseteq \{\mjx:x\in X\}$.
	
	Suppose $\mi\in \Max(P_X)$.  	Let
	$$W_1\triangleq \bigcap \{U \times V :(U,V,k)\in \mathcal{I} \}\text{ and }
	 W_2\triangleq \bigcap \{  \ua k \cap \Max(P) : (U,V,k)\in \mathcal{I} \}.$$
	We show that $W_1=W_2$. First, by the definition of $Q$, it is clear that  $W_1 \subseteq W_2$.
	Conversely, suppose $(U,V,k)\in\mathcal{I}$, i.e., $U\times V\in W_1$. Since $\mathcal I$ is a lower set,  there is $(U_0,V_0,k_0)\in \mathcal{I}$ such that $(U,V,k)\sqsubseteq (U_0,V_0,k_0)$, which follows that $U\times V\supseteq\ua k_0\cap \Max(P)\in W_2$. From this fact it follows that $W_2\subseteq W_1$. Therefore, $W_1=W_2$.
	
	Now let $D=\{k:(U,V,k)\in\mathcal{I}\}$. By using the directedness of $\mathcal{I}$, one can easily deduce that that $D$ is a directed subset of $P$. Let $k_0=\bigvee D$ as $P$ is a domain.  We have that
	$$\begin{array}{lll}
	W_1=W_2 &=& \bigcap \{\ua k \cap \Max(P) : (U,V,k)\in \mathcal{I} \}\\
	&=& \bigcap \{  \ua k : (U,V,k)\in \mathcal{I}\}\cap \Max(P)  \\
	&=&\ua k_0\cap \Max(P)\neq\emptyset.
	\end{array}
$$
	  Taking $(x,y)\in W_1$. Then,  for any $(U,V,k)\in\mathcal{I}$, we have that $x\in U$, which implies that $(U,V,k)\in \mjx$. We then have that  $\mi\subseteq \mjx$. By the maximality of $\mathcal{I}$, we obtain that $\mi=\mjx$.
	  Therefore, $\Max(P)\subseteq\{\mjx:x\in X\}$.
	
\medskip

{\bf Claim 4.} $\Max(P_X)= \{\mjx:x\in X\}$.
	
We need to show that every $\mjx$ is maximal. Note that $P_X$ is a dcpo, so $\mjx\in P_X\subseteq \da_{P_X} \Max(P_X)$. Then, there exists $\mi\in \Max(P_X)$ such that $\mjx\subseteq \mi$.
	From Claim 3, there exists $x_0\in X$ such that
	$\mi=\mj(x_0)\supseteq\mjx$. From the $T_1$-separation of $X$, it follows that
$$\{x_0\} =\bigcap\{U:(U,V,k)\in\mj(x_0)\}\subseteq \bigcap \{U:(U,V,k)\in \mjx\} =\{x\},$$
which implies that $x=x_0$. Thus, $\mjx=\mj(x_0)\in \Max(P_X)$.
 This means that $\{\mjx:x\in X\}\subseteq \Max(P_X)$. Therefore,
 $\Max(P_X)=\{\mjx:x\in X\}$.
			
\medskip

{\bf Step 4.} $P_X$ is an algebraic domain model of $X$.
			
	Define a mapping $f: X \mapsto P_X$ as follows: $\forall x\in X$,
	$$f(x)=\mjx.$$
	By Step 3, $f$ is well-defined, and one can easily check that it is bijective by using the $T_1$-separation of $X$.
			
	\medskip
	
	{\bf Claim 5.}	 $f$ is continuous.
	
	Note that $K(P_X)=\{\da_Q (U,V,k): (U,V,k)\in Q\}$, and thus $\{\ua_{P_X}(\da_Q(U,V,k))\cap \Max(P_X):(U,V,k)\in Q\}$ forms a basis for $\Max(P_X)$.
	We have that  for any $(U,V,k)\in Q$,
			\begin{equation}
				\begin{array}{lll}
					&&x\in f^{-1}(\ua_{P_X}(\da_Q(U,V,k))\cap \Max(Q)) \notag\\
					&\Leftrightarrow& f(x)=\mjx\in  \ua_{P_X}(\da_Q(U,V,k))\notag \\
					&\Rightarrow& \da_{Q}(U,V,k)\subseteq \mjx=\da\mjx\\
					&\Leftrightarrow& (U,V,k)\in \mjx\notag\\
					&\Leftrightarrow& x\in U,
					\end {array}
				\end{equation}
				which deduces that $f^{-1}( \ua_{P_X}(\da_Q(U,V,k))\cap \Max(Q)) = U\in\mathcal O(X)$. Therefore, $f$ is continuous.
	
	\medskip			
	
	{\bf Claim 6.}  $f$ is open.
	
Suppose $U\in \mathcal O(X)$ and $x_0\in X$ such that  $f(x_0)=\mj(x_0)\in f(U)$.
Since $(x_0,y_0)\in X\times Y=\Max(P)$ and $P$ is algebraic, there exists $k_0\in K(P)$ such that $(x_0,y_0)\in \ua k_0\cap \Max(P)\in \mathcal O(X\times Y)$. Then, there exists
$U_0\in\mathcal O(X)$ (we may require that $U_0\subseteq U$) and $V_0\in\mathcal O(Y)$ such that
$$(x_0,y_0)\in U_0\times V_0\subseteq \ua k_0\cap \Max(P).$$			
Clearly,  $(U_0,V_0, k_0)\in \mathcal{J}(x_0)$, and then $f(x_0)=\mj(x_0)\in \ua_{P_X}(\da_Q(U_0, V_0, k_0))\cap \Max(P_X)$.
Furthermore, for each $x\in X$,  if $f(x)=\mjx\in \ua_{P_X}(\da_{Q}(U_0,V_0,k_0))\cap \Max(P_X)$, then $(U_0,V_0,k_0)\in \mj(x)$. This implies that $x\in U_0\subseteq U$, and consequently $f(x)\in f(U_0)\subseteq f(U)$. Thus, we conclude that $f(U)$ is an open subset of $\Max (P_X)$. Therefore, $f$ is an open mapping.

\medskip

From Claims 5 and 6, it follows that $f$ is a homeomorphism.
Therefore,  $P_X$ is an algebraic domain model of $X$.			
\end{proof}

\section{Conclusions}
\begin{itemize}
		\item [(1)] Regarding Martin's problem, which asks whether there exists an ideal domain whose maximal elements do not form a $G_{\delta}$-set, we provide a positive answer. Specifically, Example \ref{exm2} demonstrates that $\Max(L)$ is not a $G_{\delta}$-set for the ideal domain $L$. However, $\Max(L)$ is (up to homeomorphism) a $G_{\delta}$-set in some other ideal domain $\widehat{L}$. This raises the following questions:
	\begin{itemize}
		\item [(i)] Is there an ideal domain $M$ such that $\Max(M)$ is not homeomorphic to a $G_{\delta}$-subset of any ideal domain?
		\item [(ii)] How can we  characterize the maximal point space of ideal domains?
	\end{itemize}
	In \cite{Brecht}, Brecht,  Goubault-Larrecq,  Jia, and Lyu asked whether every sober convergence Choquet-complete space is domain-complete (see Question (iii) in the Conclusion of \cite{Brecht}). Since we know that every ideal domain is  sober convergence Choquet-complete, a negative answer to question (i) would imply that the answer to the question in \cite{Brecht} is also negative. We leave these problems for future study.
	
	\item [(2)] In this paper, we provided a new approach to Bennett and Lutzer's problem: If $X\times Y$ is domain-representable, must $X$ and $Y$ be domain-representable?  There is another similar open problem posed in \cite{Bennett-Lutzer-2009}:
	\begin{itemize}
		\item If $X\times Y$ has a Scott domain model, must  $X$ and ${Y}$ have a Scott domain model?
	\end{itemize}						
	Here a Scott domain is a bounded complete (every up bounded subset has a supremum)  dcpo. Currently, we are still not able to solve this problem. However, the approach taken in Section 4, due to its direct construction feature, may help to find a method to answer this problem.
\end{itemize}

\medskip


\noindent{\bf Acknowledgments: }
\begin{itemize}
	\item [\rm (1)] We would like to express our gratitude to the reviewer who pointed out the reference [5], which addressed the second problem in our paper. Additionally, the reviewer provided insightful historical context on the concept of domain-representability. These contributions have significantly enriched our research and offered new perspectives.
	\item [\rm (2)] This work was supported by the National Natural Science Foundation of China (No. 12071188, 12101313).
\end{itemize}

\end{document}